\documentclass[12pt]{amsart}

\usepackage{amsmath,amssymb,amsthm,here}
\usepackage[all]{xypic}
\usepackage[dvipdfm,colorlinks=true]{hyperref}

\numberwithin{equation}{section}

\SelectTips{eu}{12}

\newtheorem{theorem}{Theorem}[section]
\newtheorem{proposition}[theorem]{Proposition}
\newtheorem{lemma}[theorem]{Lemma}
\newtheorem{corollary}[theorem]{Corollary}

\theoremstyle{definition}

\theoremstyle{remark}
\newtheorem{remark}[theorem]{Remark}

\newcommand{\Z}{\mathbb{Z}}

\newcommand{\SU}{\mathrm{SU}}

\newcommand{\Spin}{\mathrm{Spin}}
\newcommand{\G}{\mathrm{G}}
\newcommand{\F}{\mathrm{F}}
\newcommand{\E}{\mathrm{E}}
\renewcommand{\P}{\mathcal{P}}
\newcommand{\hofib}{\mathrm{hofib}}

\title{Samelson products in quasi-$p$-regular exceptional Lie groups}

\author{Sho Hasui}
\address{Faculty of Liberal Arts and Sciences, Osaka Prefecture University, Osaka, 599-8531, Japan}
\email{s.hasui@las.osakafu-u.ac.jp}

\author{Daisuke Kishimoto}
\address{Department of Mathematics, Kyoto University, Kyoto, 606-8502, Japan}
\email{kishi@math.kyoto-u.ac.jp}

\author{Toshiyuki Miyauchi}
\address{Department of Applied Mathematics, Faculty of Science, Fukuoka University, Fukuoka, 814-0180, Japan}
\email{miyauchi@math.sci.fukuoka-u.ac.jp}

\author{Akihiro Ohsita}
\address{Faculty of Economics, Osaka University of Economics, Osaka 533-8533, Japan}
\email{ohsita@osaka-ue.ac.jp}

\allowdisplaybreaks

\begin{document}

\baselineskip 16pt

\maketitle

\begin{abstract}
There is a product decomposition of a compact connected Lie group $G$ at the prime $p$, called the mod $p$ decomposition, when $G$ has no $p$-torsion in homology. Then in studying the multiplicative structure of the $p$-localization of $G$, the Samelson products of the factor space inclusions of the mod $p$ decomposition are fundamental. This paper determines (non-)triviality of these fundamental Samelson products in the $p$-localized exceptional Lie groups when the factor spaces are of rank $\le 2$, that is, $G$ is quasi-$p$-regular.
\end{abstract}


\section{Introduction}

Let $G$ be a compact connected Lie group. Recall from \cite{MNT} that if $G$ has no $p$-torsion in the integral homology, then there is a $p$-local homotopy equivalence
$$G\simeq_{(p)}B_1\times\cdots\times B_{p-1}$$
such that $B_i$ is resolvable by spheres of dimension $2i-1\mod2(p-1)$, where each $B_i$ is indecomposable if $G$ is simple except for type D. This is called the mod $p$ decomposition of $G$. For maps $\alpha\colon A\to X,\beta\colon B\to X$ into a homotopy associative H-space with inverse $X$, the composite
$$A\wedge B\xrightarrow{\alpha\wedge\beta}X\wedge X\to X$$
is called the Samelson product of $\alpha,\beta$ and is denoted by $\langle\alpha,\beta\rangle$, where the last arrow is the reduced commutator map. Then in studying the standard multiplication of the $p$-localization $G_{(p)}$, the Samelson products of the inclusions $B_i\to G_{(p)}$ are fundamental, and there are applications of these Samelson products as in \cite{M,KK,Ki}. In this paper, we aim to determine  (non-)triviality of these fundamental Samelson products in $G_{(p)}$ when $G$ is the quasi-$p$-regular exceptional Lie group, which is a continuation of the previous work \cite{HKO} on $p$-regular exceptional Lie groups.

Let us recall the result of the previous work \cite{HKO}. We say that $G$ is $p$-regular if $G$ is $p$-locally homotopy equivalent to the product of spheres. By the classical result of Hopf, $G$ is rationally homotopy equivalent to the product of spheres of dimension $2n_1-1,\ldots,2n_\ell-1$ for $n_1\le\cdots\le n_\ell$. The sequence $n_1,\ldots,n_\ell$ is called the type of $G$ and is denoted by $t(G)$. There is a list of types of simple Lie groups in \cite{KK}. It is known that when $G$ is simply connected, $G$ is $p$-regular if and only if $p$ is no less than the maximum of $t(G)$ (cf. \cite{MNT}). Obviously, if $G$ is $p$-regular, $G$ is $p$-locally homotopy equivalent to the product of spheres of dimension $2i-1$ for $i\in t(G)$. Let $\epsilon_i\colon S^{2i-1}\to G_{(p)}$ denote the inclusion for $i\in t(G)$ when $G$ is $p$-regular.

\begin{theorem}
[Hasui, Kishimoto, and Ohsita {\cite{HKO}}]
\label{HKO}
Let $G$ be the $p$-regular exceptional Lie group . The Samelson product $\langle\epsilon_i,\epsilon_j\rangle$ in $G_{(p)}$ is non-trivial if and only if there is $k\in t(G)$ such that $i+j=k+p-1$.
\end{theorem}

Let $B(2i-1,2i+2p-3)$ be the $S^{2i-1}$-bundle over $S^{2i+2p-3}$ classified by $\frac{1}{2}\alpha_1\in\pi_{2i+2p-4}(S^{2i-1})$ as in \cite{MNT,MT}, where $\alpha_1$ is a generator of the $p$-component of $\pi_{2i+2p-4}(S^{2i-1})$ which is isomorphic with $\Z/p$. Recall that $G$ is quasi-$p$-regular if $G$ is $p$-locally homotopy equivalent to the product of $B(2i-1,2i+2p-3)$'s and spheres. When $G$ is exceptional, it is shown in \cite{MT} that $G$ is quasi-$p$-regular if and only if $p\ge 5$ for $G=\G_2,\F_4,\E_6$ and $p\ge 11$ for $G=\E_7,\E_8$. In these cases, the specific mod $p$ decomposition is:

\renewcommand{\arraystretch}{1.2}
\begin{table}[H]
\label{decomp}
\centering
\begin{tabular}{lll}
$\G_2$&$p=5$&$B(3,11)$\\
&$p>5$&$S^3\times S^{11}$\\
$\F_4$&$p=5$&$B(3,11)\times B(15,23)$\\
&$p=7$&$B(3,15)\times B(11,23)$\\
&$p=11$&$B(3,23)\times S^{11}\times S^{15}$\\
&$p>11$&$S^3\times S^{11}\times S^{15}\times S^{23}$\\
$\E_6$&$p=5$&$\F_4\times B(9,17)$\\
&$p>5$&$\F_4\times S^9\times S^{17}$\\
$\E_7$&$p=11$&$B(3,23)\times B(15,35)\times S^{11}\times S^{19}\times S^{27}$\\
&$p=13$&$B(3,27)\times B(11,35)\times S^{15}\times S^{19}\times S^{23}$\\
&$p=17$&$B(3,35)\times S^{11}\times S^{15}\times S^{19}\times S^{23}\times S^{27}$\\
&$p>17$&$S^3\times S^{11}\times S^{15}\times S^{19}\times S^{23}\times S^{27}\times S^{35}$\\
$\E_8$&$p=11$&$B(3,23)\times B(15,35)\times B(27,47)\times B(39,59)$\\
&$p=13$&$B(3,27)\times B(15,39)\times B(23,47)\times B(35,59)$\\
&$p=17$&$B(3,35)\times B(15,47)\times B(27,59)\times S^{23}\times S^{39}$\\
&$p=19$&$B(3,39)\times B(23,59)\times S^{15}\times S^{27}\times S^{35}\times S^{47}$\\
&$p=23$&$B(3,47)\times B(15,59)\times S^{23}\times S^{27}\times S^{35}\times S^{39}$\\
&$p=29$&$B(3,59)\times S^{15}\times S^{23}\times S^{27}\times S^{35}\times S^{39}\times S^{47}$\\
&$p>29$&$S^3\times S^{15}\times S^{23}\times S^{27}\times S^{35}\times S^{39}\times S^{47}\times S^{59}$
\end{tabular}
\end{table}

Let $t_p(G)$ be the subset of $t(G)$ consisting of $i\in t(G)$ such that $2i-1$ is the dimension of the bottom cell of some $B_j$ in the mod $p$ decomposition of $G_{(p)}$, where $t_p(G)$ is possibly not a subset of $\{1,\ldots,p-1\}$. Since there is a one-to-one correspondence between $B_i$'s and $t_p(G)$, we ambiguously denote the factor space of $G_{(p)}$ corresponding to $i\in t_p(G)$ by $B_i$. In our case, the set $t_p(G)$ can easily be deduced from the above table as:

\renewcommand{\arraystretch}{1.2}
\begin{table}[H]
\label{t_p(G)}
\centering
\begin{tabular}{llllll}
$t_p(\G_2)$&$p=5$&$2$&$t_p(\E_7)$&$p=11$&$2,6,8,10,14$\\
$t_p(\F_4)$&$p=5$&$2,8$&&$p=13$&$2,6,8,10,12$\\
&$p=7$&$2,6$&&$p=17$&$2,6,8,10,12,14$\\
&$p=11$&$2,6,8$&$t_p(\E_8)$&$p=11$&$2,8,14,20$\\
$t_p(\E_6)$&$p=5$&$2,5,8$&&$p=13$&$2,8,12,18$\\
&$p=7$&$2,5,6,9$&&$p=17$&$2,8,12,14,20$\\
&$p=11$&$2,5,6,8,9\quad$&&$p=19$&$2,8,12,14,18,24$\\
&&&&$p=23$&$2,8,12,14,18,20$\\
&&&&$p=29$&$2,8,12,14,18,20,24$\\
\end{tabular}
\end{table}

We now state our main result, where we owe the $p$-regular case to Theorem \ref{HKO}. Let $\epsilon_i\colon B_i\to G_{(p)}$ denote the inclusion for $i\in t_p(G)$, and put $r_i=\mathrm{rank}\,B_i$.

\begin{theorem}
\label{main}
Let $G$ be the quasi-$p$-regular exceptional Lie group. Then for $i,j\in t_p(G)$, the Samelson product $\langle\epsilon_i,\epsilon_j\rangle$ in $G_{(p)}$ is trivial if and only if one of the following conditions holds:
\begin{enumerate}
\item there is no $k\in t_p(G)$ such that $i+j\equiv k\mod(p-1)$ and $i+j+(r_i+r_j-1)(p-1)> k+r_k(p-1)$;
\item $r_i+r_j\ge 3$ and there is $k\in t_p(G)$ such that $k=i+j+(r_i+r_j-3)(p-1)$;
\item $i+j=p+1$ and $r_i+r_j=3$;
\item $(G,p,\{i,j\})=(\E_6,7,\{2,6\}),(\E_7,11,\{2,8\}),(\E_7,11,\{8,10\}),(\E_8, 19, \{2,12\}),(\E_8, 19, \{12,12\})$.
\end{enumerate}
\end{theorem}

\begin{remark}
This theorem includes the result of McGibbon \cite{M} that $\G_2$ at the prime 5 is homotopy commutative. 
\end{remark}

The proof of Theorem \ref{main} consists of three parts. The first part shows triviality of the Samelson products by looking at the homotopy groups of $G$. The second part applies a criterion for non-triviality of the Samelson products by the Steenrod operations on the mod $p$ cohomology of the classifying space of $G$ which is a generalization of the criterion used to prove Theorem \ref{HKO} in \cite{HKO}. The third part determines (non-)triviality of the remaining Samelson products by considering a homotopy fibration $\hofib(\rho)\to G\xrightarrow{\rho}\SU(\infty)$ for  a stabilized representation $\rho$, where the easiest case that $\rho$ is the inclusion of $\SU(n)$ is studied in \cite{HK}. Since $\SU(\infty)$ is homotopy commutative, Samelson products lift to $\hofib(\rho)$. Then the important point is to identify the homotopy fiber $\hofib(\rho)$, and to this end, we decompose $\rho$ with respect to the mod $p$ decompositions of $G$ and $\SU(\infty)$, which is not needed in \cite{HK}. We then describe lifts of the Samelson products through the identification of $\hofib(\rho)$ and to determine (non-)triviality of the Samelson products.


\section{Triviality of Samelson products}

Hereafter we localize everything at the prime $p$.  Suppose that $(G,p)$ is as in:

\renewcommand{\arraystretch}{1.2}
\begin{table}[H]
\caption{}
\label{(G,p)}
\centering
\begin{tabular}{llll}
$\mathrm{SU}(n)$&$n\le(p-1)(p-2)+1$\\
$\mathrm{Sp}(n),\mathrm{Spin}(2n+1)$&$2n\le(p-1)(p-2)$\\
$\mathrm{Spin(2n)}$&$2(n-1)\le(p-1)(p-2)$\\
$\mathrm{G}_2,\F_4,\E_6$&$p\ge 5$\\
$\E_7,\E_8$&$p\ge 11$\\
\end{tabular}
\end{table}

It is shown in \cite{Th2} that there is a subcomplex $A_i$ of $B_i$ such that the inclusion $A_i\to B_i$ induces an isomorphism
$$\Lambda(\widetilde{H}_*(A_i))\cong H_*(B_i)$$
where $G\simeq B_1\times\cdots\times B_{p-1}$ is the mod $p$ decomposition of $G$. Put $A=A_1\vee\cdots\vee A_{p-1}$. 

\begin{theorem}
[Kishimoto {\cite{Ki}} and Theriault {\cite{Th2}}]
\label{A}
Suppose that $(G,p)$ is in Table \ref{(G,p)}. The subcomplex $A$ has the following properties:
\begin{enumerate}
\item the inclusion $\Sigma A\to\Sigma G$ has a left homotopy inverse, say $t$;
\item the inclusion $\Sigma G\to BG$ factors through its restriction $\Sigma A\to BG$ via $t$.
\end{enumerate}
\end{theorem}

Let $\epsilon_i\colon B_i\to G$ and $\bar{\epsilon}_i\colon A_i\to G$ be inclusions. 

\begin{corollary}
[Kishimoto {\cite{Ki}}]
Suppose that $(G,p)$ is in Table \ref{(G,p)}. The Samelson product $\langle\epsilon_i,\epsilon_j\rangle$ is trivial if and only if so is $\langle\bar{\epsilon}_i,\bar{\epsilon}_j\rangle$.
\end{corollary}

We then consider (non-)triviality of the Samelson products $\langle\bar{\epsilon}_i,\bar{\epsilon}_j\rangle$ instead of $\langle\epsilon_i,\epsilon_j\rangle$. We show triviality of the Samelson products by looking at the homotopy groups of spheres and $B_i$.

\begin{proposition}
[Toda \cite{To1}]
\label{pi(S)}
For $i\ge 2$ and $*\le 2i+2p(p-1)-4$, we have
$$\pi_*(S^{2i-1})\cong\begin{cases}
\Z_{(p)}&*=2i-1\\
\Z/p&*=2i-2+2j(p-1)\quad(j=1,\ldots,p-1)\\
\Z/p&*=2i-3+2j(p-1)\quad(j=i,\ldots,p-1)\\
0&\text{otherwise.}
\end{cases}$$
\end{proposition}

\begin{proposition}
[Mimura and Toda \cite{MT}, and Kishimoto \cite{Ki}]
\label{pi(B)}
For $*\le 2p(p-1)$, we have
$$\pi_*(B(3,2p+1))\cong\begin{cases}
\Z_{(p)}&*=3,3+2(p-1)\\
\Z/p&*=2j(p-1)+2\quad(j=2,\ldots,p-1)\\
0&\text{otherwise}
\end{cases}$$
and for $i\ge 3$ and $*\leq2i-4+2p(p-1)$, we have
$$\pi_*(B(2i-1,2i-1+2(p-1)))\cong\begin{cases}
\Z_{(p)}&*=2i-1,2i-1+2(p-1)\\
\Z/p^2&*=2i-2+2j(p-1)\quad(j=2,\ldots,p-1)\\
\Z/p&*=2i-3+2j(p-1)\quad(j=i,\ldots,p-1)\\
0&\text{otherwise.}
\end{cases}$$
\end{proposition}

When $G$ is a quasi-$p$-regular simple Lie group except for $\mathrm{Spin}(4n)$, there is a one-to-one correspondence between $t_p(G)$ and non-trivial $B_i$, and we have
$$A_i\simeq\begin{cases}
S^{2i-1}&r_i=1\\
S^{2i-1}\cup_{\alpha_1}e^{2i-1+2(p-1)}&r_i=2.
\end{cases}$$
for $i\in t_p(G)$. 

\begin{corollary}
\label{pi(A)}
Suppose that $G$ is a quasi-$p$-regular simple Lie group except for $\mathrm{Spin}(4n)$. For $i,j\in t_p(G)$, we have:
\begin{enumerate}
\item if there is no $k\in t_p(G)$ such that $i+j\equiv k\mod(p-1)$ and $i+j+(r_i+r_j-1)(p-1)> k+r_k(p-1)$, then the homotopy set $[A_i\wedge A_j,G]$ is trivial;
\item if there is $k\in t_p(G)$ such that $i+j\equiv k\mod(p-1)$ and $i+j+(r_i+r_j-1)(p-1)> k+r_k(p-1)$, then $k$ is unique and $[A_i\wedge A_j,G]\cong[A_i\wedge A_j,B_k]$.
\end{enumerate}
\end{corollary}

\begin{proof}
Since $A_i\wedge A_j$ for $i,j\in t_p(G)$ has cells in dimension $2i+2j-2+2r(p-1)$ for $0\le r\le r_i+r_j-2$, the corollary follows from Proposition \ref{pi(S)} and \ref{pi(B)}.
\end{proof}

\begin{corollary}
\label{Samelson-1}
Suppose that $G$ is a quasi-$p$-regular simple Lie group except for $\mathrm{Spin}(4n)$.  If for $i,j\in t_p(G)$ there is no $k\in t_p(G)$ such that 
$$i+j\equiv k\mod(p-1)\quad\text{and}\quad i+j+(r_i+r_j-1)(p-1)> k+r_k(p-1),$$
then the Samelson product $\langle\bar{\epsilon}_i,\bar{\epsilon}_j\rangle$ is trivial.
\end{corollary}

We further prove triviality of the Samelson products $\langle\bar{\epsilon}_i,\bar{\epsilon}_j\rangle$ in the special cases.

\begin{proposition}
\label{Samelson-2}
Suppose that $G$ is a quasi-$p$-regular simple Lie group except for $\mathrm{Spin}(4n)$.  If for $i,j\in t_p(G)$, there is $k\in t_p(G)$ such that 
$$r_i+r_j\ge 3\quad\text{and}\quad k=i+j+(r_i+r_j-3)(p-1),$$
then the Samelson product $\langle\bar{\epsilon}_i,\bar{\epsilon}_j\rangle$ is trivial.
\end{proposition}

\begin{proof}
By Corollary \ref{pi(A)}, we have $[A_i\wedge A_j,G]\cong[A_i\wedge A_j,B_k]$, so it is sufficient to show that $[A_i\wedge A_j,S^{2k-1}]$ is trivial since $r_k=1$ for a degree reason. We first consider the case $r_i+r_j=3$. In this case, we have $A_i\wedge A_j\simeq S^{2(i+j-1)}\cup_{\alpha_1}e^{2(i+j+p-2)}$, where $i+j=k$. Then there is a homotopy cofibration $S^{2(i+j+p-2)-1}\xrightarrow{\alpha_1}S^{2(i+j-1)}\to A_i\wedge A_j$ which induces an exact sequence
$$\pi_{2k-1}(S^{2k-1})\xrightarrow{\alpha_1^*}\pi_{2k+2p-4}(S^{2k-1})\to[A_i\wedge A_j,S^{2k-1}]\to\pi_{2k-2}(S^{2k-1}).$$
Since $\pi_{2k+2p-4}(S^{2k-1})$ is generated by $\alpha_1$, the second arrow is trivial, so for $\pi_{2k-2}(S^{2k-1})=0$, we get $[A_i\wedge A_j,S^{2k-1}]=*$.

We next consider the case $r_i=r_j=2$. In this case, we have $[A_i\wedge A_j,S^{2k-1}]\cong[A_i\wedge A_j/S^{2(i+j-1)},S^{2k-1}]$ and $A_i\wedge A_j/S^{2(i+j-1)}\simeq S^{2(i+j+p-2)}\vee(S^{2(i+j+p-2)}\cup_{\alpha_1}e^{2(i+j+2p-3)})$, where $k=i+j+p-1$. Then we get $[A_i\wedge A_j,S^{2k-1}]=*$ in the same way as above.
\end{proof}

By Corollary \ref{Samelson-1} and Proposition \ref{Samelson-2}, it remains to check (non-)triviality of the Samelson products $\langle\bar{\epsilon}_i,\bar{\epsilon}_j\rangle$ for $(G,p,\{i,j\})$ in the following table. 

\renewcommand{\arraystretch}{1.2}
\begin{table}[H]
\caption{}
\label{non-trivial1}
\centering
\begin{tabular}{lll}
$\E_6$&$p=5$&$\{2,8\},\{5,5\},\{5,8\},\{8,8\}$\\
&$p=7$&$\{2,6\},\{2,9\},\{5,6\},\{5,9\},\{6,6\},\{6,9\},\{9,9\}$\\
&$p=11$&$\{6,9\},\{8,8\},\{9,9\}$\\
$\E_7$&$p=11$&$\{2,8\},\{2,10\},\{2,14\},\{6,10\},\{6,14\},\{8,8\},\{8,10\},\{8,14\},\{10,10\},$\\
&&$\{10,14\},\{14,14\}$\\
&$p=13$&$\{2,6\},\{2,12\},\{6,6\},\{6,8\},\{6,12\},\{8,12\},\{10,10\},\{10,12\},\{12,12\}$\\
&$p=17$&$\{8,14\},\{10,12\},\{10,14\},\{12,12\},\{12,14\},\{14,14\}$\\
$\E_8$&$p=11$&$\{2,20\},\{8,14\},\{8,20\},\{14,14\},\{14,20\},\{20,20\}$\\
&$p=13$&$\{2,12\},\{2,18\},\{8,12\},\{8,18\},\{12,12\},\{12,18\},\{18,18\}$\\
&$p=17$&$\{8,20\},\{14,14\},\{14,20\},\{20,20\}$\\
&$p=19$&$\{2,12\},\{2,18\},\{2,24\},\{8,12\},\{8,18\},\{8,24\},\{12,12\},\{12,14\},\{12,18\},$\\
&&$\{12,24\},\{14,18\},\{14,24\},\{18,18\},\{18,24\},\{24,24\}$\\
&$p=23$&$\{14,20\},\{18,18\},\{20,20\}$\\
&$p=29$&$\{12,24\},\{18,18\},\{18,24\},\{20,20\},\{24,24\}$
\end{tabular}
\end{table}


\section{Cohomology of $BG$}

The $\G_2$ case and the $p$-regular case are done in the previous section and Theorem \ref{HKO}. Since the inclusion $\F_4\to\E_6$ has a right homotopy inverse at the prime $p\ge 3$ as in the table of the mod $p$ decomposition, we only consider $\E_6$ at $p=5,7,11$, $\E_7$ at $p=11,13,17$ and $\E_8$ at $p=11,13,17,19,23,29$.

The coefficient of cohomology will be $\Z_{(p)}$ unless otherwise is specified. Suppose that $H_*(G)$ has no $p$-torsion and $t(G)=\{n_1,\ldots,n_\ell\}$. Then the cohomology of the classifying space $BG$ is given by
$$H^*(BG)=\Z_{(p)}[x_{2n_1},\ldots,x_{2n_\ell}],\quad|x_i|=i.$$
We recall from \cite{W1,W2,HKO} a choice of generators $x_i$ when $G$ is the exceptional Lie group. As in \cite{A}, there is a commutative diagram of subgroup inclusions:
$$\xymatrix{\mathrm{Spin}(9)\ar[r]^{i_0}\ar[d]^{j_0}&\mathrm{Spin}(10)\ar[r]^{i_1}\ar[d]^{j_1}&\mathrm{Spin}(11)\ar[r]^{i_2}\ar[d]^{j_2}&\mathrm{Spin}(15)\ar[d]^{j_3}\\
\F_4\ar[r]^{k_0}&\E_6\ar[r]^{k_1}&\E_7\ar[r]^{k_2}&\E_8}$$
The choice of generators $x_i$ is made through these inclusions. Recall that we have
$$H^*(B\mathrm{Spin}(2n+1))=\Z_{(p)}[p_1,\ldots,p_n],\quad H^*(B\mathrm{Spin}(2n))=\Z_{(p)}[p_1,\dots,p_{n-1},c_n]$$
where $p_i$ and $c_n$ are the Pontrjagin class and the Euler class of the universal bundle respectively. If  a polynomial $P$ is a sum of a polynomial $Q$ and other terms, then we write $P\vartriangleright Q$.

\begin{proposition}
\label{x-E_8}
For $p\ge 7$, generators $x_i$ for $\E_8$ can be chosen such that
\begin{align*}
j_3^*(x_4)=&p_1\\
j_3^*(x_{16})=&12p_4-\frac{18}{5}p_3p_1+p_2^2+\frac{1}{10}p_2p_1^2\\
j_3^*(x_{24})=&60p_6-5p_5p_1-5p_4p_2+3p_3^2-p_3p_2p_1+\frac{5}{36}p_2^3\\
j_3^*(x_{28})\equiv& 480p_7+40p_5p_2-12p_4p_3-p_3p_2^2-3p_4p_2p_1+\frac{24}{5}p_3^2p_1+\frac{11}{36}p_2^3p_1\mod(p_1^2)\\
j_3^*(x_{36})\equiv& 480p_7p_2+72p_6p_3-30p_5p_4-\frac{25}{2}p_5p_2^2+ 9p_4p_3p_2-\frac{18}{5}p_3^3-\frac{1}{4}p_3p_2^3-42p_6p_2p_1+9p_5p_3p_1\\
&-\frac{3}{2}p_4p_2^2p_1+\frac{9}{5}p_3^2p_2p_1+\frac{1}{24}p_2^4p_1\mod(p_1^2)\\
j_3^*(x_{40})\equiv&480p_7p_3+50p_6p_2^2+50p_5^2-10p_5p_3p_2-\frac{25}{2}p_4^2p_2+9p_4p_3^2-\frac{25}{36}p_4p_2^3+\frac{3}{4}p_3^2p_2^2\\
&+\frac{25}{864}p_2^5\mod(p_1)\\
j_3^*(x_{48})\equiv&-200p_7p_5-60p_7p_3p_2+3p_6p_3^2+\frac{25}{9}p_6p_2^3+\frac{25}{3}p_5^2p_2-\frac{5}{2}p_5p_4p_3-\frac{25}{24}p_5p_3p_2^2-\frac{25}{48}p_4^2p_2^2\\
&+p_4p_3^2p_2+\frac{25}{864}p_4p_2^4-\frac{3}{10}p_3^4-\frac{1}{36}p_3^2p_2^3-\frac{25}{62208}p_2^6\mod(p_1)\\
j_3^*(x_{60})\vartriangleright&144p_7p_5p_3-5p_5^3+\frac{3}{2}p_5^2p_3p_2-\frac{89}{1440}p_5p_4^2p_2-\frac{229}{1600}p_5p_4p_3^2-\frac{13}{320}p_5p_4p_2^3-\frac{229}{3840}p_5p_3^2p_2^2\\
&+\frac{29}{13824}p_5p_2^5-\frac{43}{1920}p_4^3p_3+\frac{1357}{38400}p_4^2p_3p_2^2-\frac{59}{3200}p_4p_3^3p_2-\frac{421}{153600}p_4p_3p_2^4+\frac{177}{40000}p_3^5\\
&+\frac{59}{115200}p_3^3p_2^3
\end{align*}
where $j_3^*(x_{40}),j_3^*(x_{60})$ do not include multiples of $p_5p_4p_1,p_7^2p_1$ respectively
\end{proposition}

\begin{proof}
The choice of $x_i$ except for $i=60$ is made in \cite{HKO}, where we subtract a multiple of $x_4x_{36}$ from $x_{40}$ if necessary so that $j^*(x_{40})$ does not include a multiple of $p_5p_4p_1$, and we can take $x_{60}$ quite similarly. By subtracting a multiple of $x_4x_{28}^2$ if necessary, we can set that $j_3^*(x_{60})$ does not include a multiple of $p_7^2p_1$, completing the proof.
\end{proof}

\begin{proposition}
\label{x-E_7}
For $p\ge 7$, generators $x_i$ for $\E_7$ can be chosen such that:
\begin{enumerate}
\item $k_2^*(x_i)=x_i\quad(i=4,16,24,28,36)\quad j_2^*(x_{12})=-6p_3+p_2p_1\quad j_2^*(x_{20})=p_5$;
\item modulo $\widetilde{H}^*(B\E_7)^3$
\begin{alignat*}{3}
k_2^*(x_{40})\equiv&\frac{1}{24}x_{12}x_{28}+\frac{5}{24}x_{16}x_{24}+50x_{20}^2\quad&k_2^*(x_{48})\equiv&-\frac{1}{72}x_{12}x_{36}+\frac{5}{24}x_{20}x_{28}-\frac{1}{48}x_{24}^2\\
k_2^*(x_{60})\equiv&-\frac{131}{144000}x_{24}x_{36}.
\end{alignat*}
\end{enumerate}
\end{proposition}

\begin{proof}
(1) is proved in \cite{HKO}, and (2) is obtained by Proposition \ref{x-E_8} and (1).
\end{proof}

\begin{proposition}
\label{x-E_6}
For $p\ge 5$, generators $x_i$ for $\E_6$ can be chosen such that
\begin{alignat*}{4}
j_1^*(x_4)=&p_1&j_1^*(x_{10})=&c_5\\
j_1^*(x_{12})=&-6p_3+p_2p_1\quad&j_1^*(x_{16})=&12p_4-3p_3p_1+p_2^2\\
j_1^*(x_{18})=&p_2c_5&j_1^*(x_{24})=&-72p_4p_2+27p_4p_1^2+27p_3^2-9p_3p_2p_1+2p_2^3.
\end{alignat*}
\end{proposition}

\begin{proof}
The argument on the choice of  $x_i$ ($i=10,18$) for $\E_6$ in \cite{HKO} works also for $p\ge 5$, so we can choose $x_i$ ($i=10,18$) for $\E_6$ as in the statement. On the other hand, Watanabe \cite{W1} chooses generators $x_i$ for $\F_4$ through the inclusion $j_0\colon\mathrm{Spin}(9)\to\F_4$. Then since $i_0^*(p_i)=p_i$ ($i=1,2,3,4$) and $i_0^*(c_5)=0$, a degree reason shows that the choice of $x_i$ for $\F_4$ implies the choice of $x_i$ for $\E_6$  ($i=4,12,16,24$). 
\end{proof}

\begin{remark}
We choose generators $x_i$ for $\E_6$ independently from $\E_7,\E_8$ since we have to consider the primes $5,7$. 
\end{remark}


\section{Steenrod operations and Samelson products}

Suppose that $(G,p)$ is in Table \ref{(G,p)} except for $\mathrm{Spin}(4n)$, where we exclude $\mathrm{Spin}(4n)$ to make $t(G)$ consist of distinct integers. Define $y_{2j-1}\in H^{2j-1}(A_i)$ by $(\bar{\epsilon}_i')^*(x_{2j})=\Sigma y_{2j-1}$ for $j\equiv i\mod (p-1)$, where $f'\colon\Sigma X\to Y$ denotes the adjoint of a map $f\colon X\to\Omega Y$. Then $y_{2j-1}$ is non-trivial and satisfies
$$(\bar{\epsilon}_i')^*(x_{2k})=\begin{cases}\Sigma y_{2k-1}&k\equiv i\mod(p-1)\\0&k\not\equiv i\mod(p-1)\end{cases}$$
since $t(G)$ consists of distinct integers. We detect non-triviality of the Samelson products $\langle\bar{\epsilon}_i,\bar{\epsilon}_j\rangle$ by the following criterion. (cf. \cite[Proof of Theorem 1.1]{KO}, \cite{KT})

\begin{proposition}
\label{P^1}
Suppose that $(G,p)$ is in Table \ref{(G,p)} except for $\mathrm{Spin}(4n)$ and that for $i,j\in t_p(G)$, there is $k\in t_p(G)$ such that $i+j>k$ and $\P^{r_k}x_{2k}$ includes $\lambda x_{2i+2s_i(p-1)}x_{2j+2s_j(p-1)}$ with $\lambda\ne 0,s_i\le\min\{r_i-1,r_k-1\},s_j\le\min\{r_j-1,r_k-1\}$. Then $\langle\bar{\epsilon}_i,\bar{\epsilon}_j\rangle$ is non-trivial.
\end{proposition}

\begin{proof}
Assume that $\langle\bar{\epsilon}_i,\bar{\epsilon}_j\rangle$ is trivial. Then by the adjointness of Samelson products and Whitehead products, the Whitehead product $[\bar{\epsilon}_i',\bar{\epsilon}_j']$ is trivial, implying 
$\bar{\epsilon}_i'\vee\bar{\epsilon}_j'\colon\Sigma A_i\vee\Sigma A_j\to BG$ extends to $\mu\colon\Sigma A_i\times\Sigma A_j\to BG$ up to homotopy. Let $\bar{\mu}$ be the restriction of $\mu$ to $\Sigma A^{{(2i-1+2(r_k-1)(p-1))}}\times\Sigma A^{{(2j-1+2(r_k-1)(p-1))}}$. Then we have $\mathcal{P}^{r_k}\bar{\mu}^*(x_{2k})=0$. On the other hand, we have
$$\mathcal{P}^{r_k}\bar{\mu}^*(x_{2k})=\bar{\mu}^*(\mathcal{P}^{r_k}x_{2k})=\bar{\mu}^*(\lambda x_{2i+2s_i(p-1)}x_{2j+2s_j(p-1)})=\lambda\Sigma y_{2i+2s_i(p-1)}\otimes\Sigma y_{2j+2s_j(p-1)}$$
since $\P^{r_k}x_{2k}$ has no linear part for a degree reason. This is a contradiction, so the Samelson product $\langle\bar{\epsilon}_i,\bar{\epsilon}_j\rangle$ is non-trivial.
\end{proof}

In order to apply Proposition \ref{P^1}, we calculate the linear and the quadratic parts of $\P^1x_{2k}$.

\begin{lemma}
\label{P^1-E_6}
The linear and the quadratic parts of $\P^1x_i$ for $\E_6$ are given by:
\renewcommand{\arraystretch}{1.2}
\begin{table}[H]
\centering
\begin{tabular}{lllllll}
&$i=4$&$i=10$&$i=12$&$i=16$&$i=18$&$i=24$\\
$p=5$&$-x_{12}$&$-x_{18}$&$0$&$x_{24}$&$-x_{10}x_{16}$&$x_{16}^2$\\
$p=7$&$-2x_{16}+5x_4x_{12}$&$x_4x_{18}+3x_{10}x_{12}$&$-2x_{24}-4x_{12}^2$&$-3x_{12}x_{16}$&$-x_{12}x_{18}$&$-2x_{18}^2$\\
$p=11$&$-2x_{24}-2x_{12}^2$&$-x_{12}x_{18}$&$-4x_{16}^2$&$6x_{18}^2$&$0$&$0$
\end{tabular}
\end{table}
\end{lemma}

\begin{proof}
Recall from \cite{S}  that there is the mod $p$ Wu formula
\begin{multline*}
\P^1p_n=\sum_{i_1+2i_2+\cdots+5i_5=n+\frac{p-1}{2}}(-1)^{i_1+\cdots+i_5+\frac{p+1}{2}}\frac{(i_1+\cdots+i_5-1)!}{i_1!\cdots i_5!}\\
\times\left(2n-1-\frac{\sum_{j=1}^{n-1}(2n+p-1-2j)i_j}{i_1+\cdots+i_5-1}\right)p_1^{i_1}\cdots p_5^{i_5}
\end{multline*}
in $H^*(B\mathrm{Spin}(10);\Z/p)$, where $p_5=c_5^2$. Then by Proposition \ref{x-E_6} we have
$$j_1^*(\P^1x_{16})=\P^1j^*(x_{16})\vartriangleright 7p_5p_4+2p_5p_2^2,$$
so for a degree reason, we must have $\P^1x_{16}\vartriangleright 6x_{18}^2+7x_{16}x_{10}^2$. The remaining calculation is done in the same way.
\end{proof}

\begin{lemma}
\label{P^1-E_8}
The linear and the quadratic parts of $\P^1x_i$ for $\E_8$ are given by:
\renewcommand{\arraystretch}{1.2}
\begin{table}[H]
\centering
\begin{tabular}{lllll}
&$i=4$&$i=16$&$i=24$&$i=28$\\
&$i=36$&$i=40$&$i=48$&$i=60$\\
$p=11$&$2x_{24}$&$6x_{36}$&$0$&$3x_{48}$\\
&$3x_{16}x_{40}+9x_{28}^2$&$9x_{60}$&$5x_{28}x_{40}$&$5x_{40}^2$\\
$p=13$&$-x_{28}+8x_4x_{24}$&$8x_{40}-2x_4x_{36}+4x_{16}x_{24}$&$4x_{48}+5x_{24}^2$&$5x_{24}x_{28}$\\
&$5x_{60}+2x_{24}x_{36}$&$9x_4x_{60}-x_{16}x_{48}$&$-x_{24}x_{48}-x_{36}^2$&$8x_{24}x_{60}+7x_{36}x_{48}$\\
&&$+x_{24}x_{40}+4x_{28}x_{36}$\\
$p=17$&$4x_{36}$&$13x_{48}+2x_{24}^2$&$11x_{16}x_{40}+7x_{28}^2$&$12x_{60}+5x_{24}x_{36}$\\
&$0$&$-x_{24}x_{48}+10x_{36}^2$&$13x_{40}^2$&$0$\\
$p=19$&$4x_{40}+11x_4x_{36}$&$10x_4x_{48}+11x_{16}x_{36}$&$10x_{60}+x_{24}x_{36}$&$-x_4x_{60}+10x_{16}x_{48}$\\
&$+9x_{16}x_{24}$&$+17x_{24}x_{28}$&&$-3x_{24}x_{40}-4x_{28}x_{36}$\\
&$4x_{24}x_{48}+4x_{36}^2$&$5x_{16}x_{60}+x_{36}x_{40}$&$11x_{24}x_{60}+9x_{36}x_{48}$&$-3x_{36}x_{60}-2x_{48}^2$\\
$p=23$ &$10x_{48}+x_{24}^2$&$x_{60}+5x_{24}x_{36}$&$-9x_{28}x_{40}$&$8x_{24}x_{48}-10x_{36}^2$\\
&$13x_{40}^2$&$3x_{24}x_{60}-8x_{36}x_{48}$&$0$&$0$\\
$p=29$&$-2x_{60}-x_{24}x_{36}$&$-x_{24}x_{48}-5x_{36}^2$&$11x_{40}^2$&$17x_{24}x_{60}-5x_{36}x_{48}$\\
&$0$&$14x_{36}x_{60}+13x_{48}^2$&$0$&$0$
\end{tabular}
\end{table}
\end{lemma}

\begin{proof}
The proof is the same as Lemma \ref{P^1-E_6}.
\end{proof}

\begin{lemma}
\label{P^1-E_7}
The linear and the quadratic parts of $\P^1x_i$ for $\E_7$ are given by:
\renewcommand{\arraystretch}{1.2}
\begin{table}[H]
\centering
\begin{tabular}{lllll}
&$i=4$&$i=12$&$i=16$&$i=20$\\
&$i=24$&$i=28$&$i=36$\\
$p=11$&$2x_{24}$&$-3x_4x_{28}-5x_{12}x_{20}$&$6x_{36}$&$-3x_4x_{36}-5x_{12}x_{28}$\\
&&$-2x_{16}^2$&&$+10x_{20}^2$\\
&$0$&$5x_{12}x_{36}+2x_{20}x_{28}$&$9x_{28}^2$\\
&&$+2x_{24}^2$\\
$p=13$&$-x_{28}+8x_4x_{24}$&$-2x_{36}+7x_{12}x_{24}$&$-2x_4x_{36}-4x_{12}x_{28}$&$7x_{16}x_{28}+4x_{20}x_{24}$\\
&&$+2x_{16}x_{20}$&$-3x_{16}x_{24}-3x_{20}^2$\\
&$5x_{12}x_{36}+3x_{20}x_{28}+6x_{24}^2$&$5x_{24}x_{28}$&$-3x_{24}x_{36}$\\
$p=17$&$4x_{36}$&$5x_{16}x_{28}-2x_{20}x_{24}$&$-x_{12}x_{36}+2x_{20}x_{28}+6x_{24}^2$&$7x_{24}x_{28}$\\
&$7x_{28}^2$&$-3x_{24}x_{36}$&$0$
\end{tabular}
\end{table}
\end{lemma}

\begin{proof}
$\P^1x_i$ for $i=4,16,24,28,36$ can be calculated by Proposition \ref{x-E_7} and Lemma \ref{P^1-E_8}, and $\P^1x_i$ for $i=12,20$ can be calculated in the same way as Lemma \ref{P^1-E_6}.
\end{proof}

We now prove:

\begin{proposition}
\label{Samelson-3}
The Samelson products $\langle\bar{\epsilon}_i,\bar{\epsilon}_j\rangle$ in $G$ is non-trivial for $(G,p,\{i,j\})$ in the following table.
\renewcommand{\arraystretch}{1.2}
\begin{table}[H]
\centering
\begin{tabular}{lll}
$\E_6$&$p=5$&$\{5,8\},\{8,8\}$\\
&$p=7$&$\{2,9\},\{5,6\},\{6,6\},\{6,9\},\{9,9\}$\\
&$p=11$&$\{6,9\},\{8,8\},\{9,9\}$\\
$\E_7$&$p=11$&$\{2,14\},\{6,10\},\{6,14\},\{8,8\},\{10,10\},\{10,14\},\{14,14\}$\\
&$p=13$&$\{8,12\},\{10,10\},\{10,12\},\{12,12\}$\\
&$p=17$&$\{8,14\},\{10,12\},\{10,14\},\{12,12\},\{12,14\},\{14,14\}$\\
$\E_8$&$p=11$&$\{8,20\},\{14,14\},\{14,20\},\{20,20\}$\\
&$p=13$&$\{2,18\},\{8,12\},\{12,12\},\{12,18\},\{18,18\}$\\
&$p=17$&$\{8,20\},\{14,14\},\{20,20\}$\\
&$p=19$&$\{2,24\},\{8,12\},\{8,18\},\{8,24\},\{12,14\},\{12,18\},\{12,24\},\{14,18\},\{18,18\},$\\
&&$\{18,24\},\{24,24\}$\\
&$p=23$&$\{14,20\},\{18,18\},\{20,20\}$\\
&$p=29$&$\{12,24\},\{18,18\},\{18,24\},\{20,20\},\{24,24\}$
\end{tabular}
\end{table}
\end{proposition}

\begin{proof}
We can verify the conditions of Proposition \ref{P^1} by Lemma \ref{P^1-E_6}, \ref{P^1-E_8} and \ref{P^1-E_7}, where we have $\mathcal{P}^1\mathcal{P}^1=2\mathcal{P}^2$ by the Adem relation. Thus the result follows from Proposition \ref{P^1}.
\end{proof}


\section{Chern classes}

In order to determine (non-)triviality of the Samelson products that are not detected in the previous sections, we will use representations of the exceptional Lie groups and their Chern classes. Then we calculate these Chern classes. We refer to \cite{A} for basic materials of representations that we consider in this section. For the canonical representation $\lambda_n\colon\mathrm{Spin}(n)\to\SU(n)$, we have 
\begin{equation}
\label{canonical-rep}
c_{2i-1}(\lambda_n)=0,\quad c_{2i}(\lambda_n)=(-1)^ip_i\quad(i=2,\ldots,n)
\end{equation}
where $p_{2k}=c_k^2$ if $n=2k$. Then by Girard's formula on power sums and elementary symmetric polynomials
\begin{equation}
\label{Girard}
n!\mathrm{ch}_n=\sum_{i_1+2i_1+\cdots+ni_n=n}(-1)^{n+i_1+\cdots+i_n}\frac{n(i_1+\cdots+i_n-1)!}{i_1!\cdots i_n!}c_1^{i_1}\cdots c_n^{i_n},
\end{equation}
we can calculate the Chern character of $\lambda$, where $\mathrm{ch}_n$ denotes the $2n$-dimensional part of the Chern character. Let $\alpha\colon\mathrm{Spin}(11)\to\SU(55)$ be the the adjoint representation of $\mathrm{Spin}(11)$, and let $\Delta^+,\Delta$ be the positive half spin representation of $\mathrm{Spin}(10)$ and the spin representation of $\mathrm{Spin}(11)$. The weights of $\alpha$ are the roots of $\mathrm{Spin}(11)$ by definition. As in \cite{A}, the weights of $\Delta^+$ are $\epsilon_1t_1+\cdots+\epsilon_5t_5$ ($\epsilon_1\cdots\epsilon_5=1$) and the weights of $\Delta$ are $\epsilon_1t_1+\cdots+\epsilon_5t_5$ ($\epsilon_1\cdots\epsilon_5=\pm 1$). Then one can calculate $\mathrm{ch}(\alpha),\mathrm{ch}(\Delta),\mathrm{ch}(\Delta^+)$ with an assistance of a computer as follows.

\begin{lemma}
\label{rep-Spin}
\begin{enumerate}
\item $i!\mathrm{ch}_i(\lambda_{10}+\Delta^++1)$ includes the following terms:
\renewcommand{\arraystretch}{1.2}
\begin{table}[H]
\centering
\begin{tabular}{llllllll}
$i=2$&$6p_1$&$i=5$&$60c_5$&$i=6$&$18p_3$\\
$i=8$&$60p_4+24p_3p_1$&$i=9$&$126p_2c_5$&$i=10$&$135p_4p_1+630c_5^2$\\
$i=12$&$135p_4p_2+18p_3^2$&$i=14$&$231p_4p_3+\frac{1337}{4}p_4p_2p_1+2233p_2c_5^2$&$i=20$&$\frac{6885}{4}p_4^2p_2$
\end{tabular}
\end{table}
\item $i!\mathrm{ch}_i(2\lambda_{11}+\Delta+2)$ includes the following terms:
\renewcommand{\arraystretch}{1.2}
\begin{table}[H]
\centering
\begin{tabular}{llll}
$i=2$&$12p_1$&$i=6$&$36p_3$\\
$i=8$&$120p_4+48p_3p_1$&$i=10$&$1260p_5+270p_4p_1$\\
$i=12$&$270p_4p_2+36p_3^2$&$i=14$&$4466p_5p_2+462p_4p_3+\frac{1337}{2}p_4p_2p_1$\\
$i=18$&$39672p_5p_4+6993p_5p_2^2+1134p_4p_3p_2$&$i=20$&$151300p_5^2+189190p_5p_4p_1+37570p_5p_3p_2$\\
&&&$+\frac{50785}{2}p_5p_2^2p_1+\frac{6885}{2}p_4^2p_2$\\
$i=22$&$179949p_5p_4p_2+\frac{27797}{4}p_5p_2^3$&$i=24$&$950400p_5^2p_2+390390p_5p_4p_3+\frac{19569}{2}p_4^3$\\
&&&$+\frac{43887}{8}p_4^2p_2^2$\\
$i=30$&$\frac{27047655}{2}p_5^2p_3p_2+\frac{3907395}{2}p_5p_4p_3^2$\\
&$+\frac{2450295}{8}p_5p_3^2p_2^2$
\end{tabular}
\end{table}
\item $i!\mathrm{ch}_i(\alpha+4\Delta+65)$ includes the following terms:
\renewcommand{\arraystretch}{1.2}
\begin{table}[H]
\centering
\begin{tabular}{llll}
$i=2$&$60p_1$&$i=8$&$1440p_4$\\
$i=10$&$0\cdot p_5$&$i=12$&$-7560p_4p_2$\\
$i=14$&$92400p_5p_2-40110p_4p_2p_1$&$i=18$&$-982800p_5p_4$\\
$i=20$&$4600200p_5^2-1748790p_4^2p_2$&$i=22$&$34950300p_5^2p_1-6715170p_4^2p_3$\\
$i=24$&$-69872880p_5^2p_2+20077794p_4^3$&$i=26$&$-219540750p_5p_4^2-41441400p_4p_3^3$\\
&$+\frac{22514031}{2}p_4^2p_2^2$&&$-697554000p_5^2p_2p_1$\\
$i=30$&$-2289787500p_5^3+\frac{10586291625}{2}p_5p_4^2p_2$&$i=32$&$-26808164160p_5^3p_1-29632951680p_5^2p_4p_2$\\
&$-\frac{2479051575}{2}p_4^3p_3$&&$+1801812456p_5^2p_3^2-1601056128p_5p_3^3p_2$\\
$i=34$&$66371012400p_5^3p_2$&$i=38$&$-952563046800p_5^3p_4-949011128850p_5^3p_2^2$\\
&&&$-46394357586p_5^2p_3^3$\\
$i=42$&$-1030173212250p_3^2p_5^3$&$i=50$&$-914425331875000 p_5^5$
\end{tabular}
\begin{tabular}{ll}
$i=60$&$-340771201982677620p_5^5p_3p_2-12363661137209454345p_5^4p_4^2p_2+776927112035890410p_5^4p_4p_3^2$\\
&$-\frac{20464209777645659655}{2}p_5^4p_4p_2^3+\frac{4154188924169320995}{4}p_5^4p_3^2p_2^2-\frac{15713809488581615145}{16}p_5^4p_2^5$\\
&$-\frac{12852298085402204835}{8}p_5^3p_4p_3^3p_2+\frac{44662086886161465}{2}p_5^3p_3^5-\frac{15212346544088595405}{32}p_5^3p_3^3p_2^3$\\
&$+\frac{16285084675436347155}{8}p_5^2p_4^5+\frac{634726068898356739815}{64}p_5^2p_4^4p_2^2-\frac{226812497505304105395}{32}p_5^2p_4^3p_3^2p_2$\\
&$+\frac{149364900987340422405}{32}p_5^2p_4^3p_2^4+\frac{5425603982540815815}{8}p_5^2p_4^2p_3^4-\frac{658004024293760076975}{128}p_5^2p_4^2p_3^2p_2^3$\\
&$+\frac{227646297744518937585}{512}p_5^2p_4^2p_2^6+\frac{16657727505432381165}{16}p_5^2p_4p_3^4p_2^2-\frac{236740660412585333865}{512}p_5^2p_4p_3^2p_2^5$\\
&$+\frac{19109048218262101365}{2048}p_5^2p_4p_2^8-32085165434973420p_5^2p_3^6p_2+\frac{10050849772724723655}{128}p_5^2p_3^4p_2^4$\\
&$-\frac{126872100695409424425}{512}p_4^7p_2-\frac{561833719261457328105}{2048}p_4^6p_2^3-\frac{417106897597216552845}{8192}p_4^5p_2^5. $
\end{tabular}
\end{table}
\end{enumerate}
\end{lemma}

\begin{remark}
For the case $(i,\ell)=(60,8)$, the twenty-four terms in the above list are chosen so that the following condition holds. For a monomial $m\in H^{120}(B\Spin(11))/(p_1)$, let $a_m$ be the vector consisting of the coefficients of $m$ in the monomials of $H^{120}(B\E_8)/(x_4)$. Then $a_m$ for the above monomials form a square matrix, which is invertible.
Note that $a_{p_5^6}$ and $a_{p_5^5p_3p_2}$ are linearly dependent.
\end{remark}

Let $\rho_\ell$ be the irreducible 27, 56, 248 dimensional representation of $\E_\ell$ for $\ell=6,7,8$ respectively. Then we have
$$\rho_6\circ j_1=\lambda_{10}+\Delta^++1,\quad\rho_7\circ j_2=2\lambda_{11}+\Delta+2,\quad\rho_8\circ j_3\circ i_2=\alpha+4\Delta+65.$$
Thus by Proposition \ref{x-E_8}, \ref{x-E_7}, \ref{x-E_6} and Lemma \ref{rep-Spin}, we can determine the linear and the quadratic parts of $\mathrm{ch}_i(\rho_\ell)$ except for the coefficient of $x_{36}x_{48}$ in $\mathrm{ch}_{42}(\rho_8)$. Then by the inductive use of \eqref{Girard}, we obtain the following proposition, which gives the linear and the quadratic parts of $c_i(\rho_\ell)$ in each case except for $(i,\ell)=(42,8)$.

\begin{proposition}
\label{Chern-1}
\begin{enumerate}
\item The linear and the quadratic parts of the Chern classes $c_i(\rho_6)$ are:
\renewcommand{\arraystretch}{1.2}
\begin{table}[H]
\centering
\begin{tabular}{llllllll}
$i=2$&$-3x_4$&$i=5$&$12x_{10}$&$i=6$&$\frac{1}{2}x_{12}$\\
$i=8$&$-\frac{5}{8}x_{16}-\frac{11}{16}x_4x_{12}$&$i=9$&$14x_{18}$&$i=10$&$\frac{3}{4}x_4x_{16}+9x_{10}^2$\\
$i=12$&$\frac{5}{32}x_{24}-\frac{13}{384}x_{12}^2$&$i=14$&$-\frac{19}{192}x_4x_{24}+\frac{17}{2}x_{10}x_{18}-\frac{1}{12}x_{12}x_{16}$&$i=20$&$\frac{1}{512}x_{16}x_{24}$
\end{tabular}
\end{table}
\item The linear and the quadratic parts of the Chern classes $c_i(\rho_7)$ are:
\renewcommand{\arraystretch}{1.2}
\begin{table}[H]
\centering
\begin{tabular}{llll}
$i=2$&$-6x_4$&$i=6$&$x_{12}$\\
$i=8$&$-\frac{5}{4}x_{16}-\frac{17}{4}x_4x_{12}$&$i=10$&$-126x_{20}+\frac{21}{4}x_4x_{16}$\\
$i=12$&$\frac{9}{2}x_{24}+\frac{907}{2}x_4x_{20}+\frac{1}{24}x_{12}^2$&$i=14$&$-\frac{319}{40}x_{28}-\frac{67}{8}x_4x_{24}+\frac{43}{80}x_{12}x_{16}$\\
$i=18$&$\frac{1229}{60}x_{36}-\frac{749}{200}x_{12}x_{24}+\frac{601}{24}x_{16}x_{20}$&$i=20$&$-\frac{1043}{24}x_4x_{36}-\frac{71}{480}x_{12}x_{28}-\frac{441}{160}x_{16}x_{24}+373x_{20}^2$\\
$i=22$&$-\frac{137}{640}x_{16}x_{28}+\frac{1827}{20}x_{20}x_{24}$&$i=24$&$-\frac{1711}{1440}x_{12}x_{36}+\frac{297}{20}x_{20}x_{28}+\frac{4047}{800}x_{24}^2$\\
$i=30$&$\frac{963}{400}x_{24}x_{36}$
\end{tabular}
\end{table}
\item The linear and the quadratic parts of the Chern classes $c_i(\rho_8)$ are:
\renewcommand{\arraystretch}{1.2}
\begin{table}[H]
\centering
\begin{tabular}{llll}
$i=2$&$-30x_4$&$i=8$&$-15x_{16}$\\
$i=12$&$-126x_{24}$&$i=14$&$-165x_{28}+3306x_4x_{24}$\\
$i=18$&$-1820x_{36}$&$i=20$&$-\frac{23001}{5}x_{40}+\frac{408519}{10}x_4x_{36}+\frac{27821}{20}x_{16}x_{24}$\\
$i=22$&$106233x_4x_{40}+\frac{5685}{16}x_{16}x_{28}$&$i=24$&$\frac{1746822}{5}x_{48}+\frac{120512}{25}x_{24}^2$\\
$i=26$&$-7648542x_4x_{48}+\frac{92275}{24}x_{16}x_{36}$&$i=30$&$-15265250x_{60}-\frac{24568999}{576}x_{24}x_{36}$\\
&$+\frac{18645}{4}x_{24}x_{28}$\\
$i=32$&$\frac{1236701415}{4}x_{4}x_{60}-\frac{747928}{5}x_{16}x_{48}$&$i=34$&$-\frac{2170113}{10}x_{28}x_{40}$\\
&$+\frac{40849521}{200}x_{24}x_{40}+\frac{29810159}{960}x_{28}x_{36}$\\
$i=38$&$-\frac{230904475}{12}x_{16}x_{60}+\frac{16714837}{20}x_{28}x_{48}$&$i=42$&$68678214150x_{24}x_{60}$\\
&$+\frac{95373791}{60}x_{36}x_{40}$\\
$i=50$&$-2930823500 x_{40}x_{60}$&$i=60$&$-\frac{23018190805225}{48}x_{60}^2$
\end{tabular}
\end{table}
\end{enumerate}
\end{proposition}


Since we are computing $c_i(\rho_\ell)$ via $\mathrm{Spin}(10)$ and $\mathrm{Spin}(11)$ whose ranks are less than $\mathrm{E}_\ell$, $c_i(\rho_\ell)$ might not be determined in some cases by the above direct computation. In these cases, we determine $c_i(\rho_\ell)\mod p$ by an indirect way as follows.

\begin{proposition}
\label{Chern-2}
The quadratic parts of $c_i(\rho_\ell)\mod p$ are given by:
\renewcommand{\arraystretch}{1.2}
\begin{table}[H]
\centering
\begin{tabular}{llll}
$(i,\ell,p)=(28,7,11)$&$3x_{20}x_{36}-2x_{28}^2$&$(i,p,\ell)=(36,7,13)$&$-4x_{36}^2$\\
\end{tabular}
\end{table}
\end{proposition}

\begin{proof}
We only calculate $c_{28}(\rho_7)\mod 11$ since the other case can be similarly calculated. Recall from \cite{S} that there is the mod $p$ Wu formula
\begin{multline*}
\P^1c_k=\sum_{i_1+2i_2+\cdots+ni_n=k+p-1}(-1)^{i_1+\cdots+i_n-1}\frac{(i_1+\cdots+i_n-1)!}{i_1!\cdots i_n!}\\
\times\left(k-1-\frac{\sum_{j=2}^{k-1}(k+p-1-j)i_j}{i_1+\cdots+i_n-1}\right)c_1^{i_1}\cdots c_n^{i_n}
\end{multline*}
in $H^*(B\mathrm{U}(n);\Z/p)$. Then we have $\P^1c_{18}\vartriangleright 6c_{28}+c_{10}c_{18}$ and $6c_{28}(\rho_7)+c_{10}(\rho_7)c_{18}(\rho_7)=6c_{28}(\rho_7)-8x_{20}x_{36}$ by Proposition \ref{Chern-1}.
On the other hand, by Lemma \ref{P^1-E_7} and Proposition \ref{Chern-1}, we have $\P^1c_{18}(\rho_6)\equiv\P^1(\tfrac{1229}{60}x_{36}+\tfrac{601}{24}x_{16}x_{20})\equiv-x_{28}^2-x_{20}x_{36}\mod\widetilde{H}^*(B\E_7;\Z/11)^3$. Then we obtain the desired result.
\end{proof}


\section{Decomposition of representations}

In order to calculate the Samelson products, we will need to identify the homotopy fiber of a stabilized representation of $G$. To this end, we decompose stabilized representations with respect to the mod $p$ decomposition of $G$. Suppose that $G$ is in Table \ref{(G,p)}. Then as in Section 2, there is a homology generating subcomplex $A$ of $G$. We recall some more properties of $A$ that we are going to use. The following property is proved by looking carefully at the construction of Cohen and Neisendorfer \cite{CN}. For a map $f\colon X\to Y$ where $Y$ is a homotopy associative H-space, we denote its extension $\Omega\Sigma X\to Y$ by $\bar{f}$. As in \cite{Th1}, there is a homotopy associative and homotopy commutative H-structure of each $B_i$ satisfying certain properties. We first recall these properties. The product of the above H-structures of $B_i$'s defines an H-structure of $G$ which we call the exotic H-structure since it is different from the standard one in general. By definition, the exotic H-structure of $G$ is homotopy associative.

\begin{proposition}
[Theriault \cite{Th2}]
\label{retraction}
Suppose $(G,p)$ is in Table \ref{(G,p)}. There is a map $r\colon\Omega\Sigma A\to G$ such that:
\begin{enumerate}
\item the composite $G\xrightarrow{E}\Omega\Sigma G\xrightarrow{\Omega t}\Omega\Sigma A\xrightarrow{r}G$ is homotopic to the identity map;
\item for any map $f\colon A\to Z$ into a homotopy associative and homotopy commutative H-space $Z$, there is an $H$-map $g\colon G\to Z$ with respect to the exotic H-structure of $G$ such that $g\circ r\simeq\bar{f}$.
\end{enumerate}
\end{proposition}

\begin{proposition}
[Theriault \cite{Th2}]
\label{universality}
Suppose $(G,p)$ is in Table \ref{(G,p)}. Let $f_1,f_2\colon G\to Y$ be H-maps with respect to the exotic H-structure of $G$, where $Y$ is a homotopy associative and homotopy commutative H-space. Then $f_1\simeq f_2$ if and only if $f_1\vert_A\simeq f_2\vert_A$.
\end{proposition}


We consider the alternation of the standard H-structure and the exotic H-structure of $G$.

\begin{lemma}
\label{exotic}
Suppose that $(G,p)$ is in Table \ref{(G,p)}. Let $f\colon G\to Z$ be an H-map with respect to the standard H-structure of $G$ where $Z$ is a homotopy associative and homotopy commutative H-space. Then $f$ is also an H-map with respect to the exotic H-structure of $G$.
\end{lemma}

\begin{proof}
Let $h\colon\Omega\Sigma A\to G$ be the extension of the inclusion, where we consider the standard H-structure of $G$. Then since $f$ is an H-map, the composite $f\circ h$ is also an H-map which restricts to $f\vert_A$. Then by universality of loop-suspensions, we get $f\circ h\simeq\overline{f\vert_A}$. By Theorem \ref{A}, the composite $G\xrightarrow{E}\Omega\Sigma G\xrightarrow{\Omega t}\Omega\Sigma A\xrightarrow{h}G$ is homotopic to the identity map of $G$, so the composite $G\xrightarrow{E}\Omega\Sigma G\xrightarrow{\Omega t}\Omega\Sigma A\xrightarrow{\overline{f\vert_A}}Z$ is homotopic to $f$. On the other hand, by Proposition \ref{retraction}, there is an H-map $g\colon G\to Z$ with respect to the exotic H-structure of $G$ such that the composite $\Omega\Sigma A\xrightarrow{r}G\xrightarrow{g}Z$ is homotopic to $\overline{f\vert_A}$. Then we obtain
$$f
\simeq\overline{f\vert_A}\circ\Omega t\circ E\simeq g\circ r\circ\Omega t\circ E\simeq g,$$
completing the proof. 
\end{proof}

Let $\SU(\infty)\simeq C_1\times\cdots\times C_{p-1}$ be the mod $p$ decomposition such that $\pi_*(C_k)=0$ for $*\not\equiv 2k+1\mod 2(p-1)$. Let $\widetilde{C}_k$ and $\widetilde{\SU}(\infty)$ be the 3-connective covers of $C_k$ and $\SU(\infty)$, respectively. Then we have $\widetilde{SU}(\infty)\simeq\widetilde{C}_1\times\cdots\times\widetilde{C}_{p-1}$ and there is a homotopy equivalence $\SU(\infty)\xrightarrow{\simeq}\Omega^2\widetilde{SU}(\infty)$ which is a loop map. We now decompose an H-map $\rho\colon G\to\SU(\infty)$ with respect to the mod $p$ decompositions of $G$ and $\SU(\infty)$. Define a map $\rho^k\colon B_i\to\Omega^2\widetilde{C}_k$ by the composite
$$B_k\xrightarrow{\rm incl}G\xrightarrow{\rho}\SU(\infty)\xrightarrow{\simeq}\Omega^2\widetilde{\SU}(\infty)\xrightarrow{\rm proj}\Omega^2\widetilde{C}_k.$$

\begin{lemma}
\label{rho^i}
Suppose that $(G,p)$ is in Table \ref{(G,p)}. For an H-map $\rho\colon G\to\SU(\infty)$ with respect to the standard H-structure of $G$, the composite
$$G\simeq B_1\times\cdots\times B_{p-1}\xrightarrow{\rho^1\times\cdots\times\rho^{p-1}}\Omega^2\widetilde{C}_1\times\cdots\times\Omega^2\widetilde{C}_{p-1}\simeq\SU(\infty)$$
is an H-map with respect to the exotic H-structure of $G$.
\end{lemma}

\begin{proof}
By Lemma \ref{exotic}, the map $\rho$ is an H-map with respect to the exotic H-structure of $G$, and by definition, the inclusion $B_k\to G$ is an H-map with respect to the exotic H-structure of $G$. Then by definition, each $\rho^k$ is an H-map.
\end{proof}

\begin{theorem}
\label{decomp-rho}
Suppose that $(G,p)$ is in Table \ref{(G,p)}. If $\rho\colon G\to\SU(\infty)$ is an H-map with respect to the standard H-structure, then $\rho\simeq\rho^1\times\cdots\times\rho^{p-1}$.
\end{theorem}

\begin{proof}
By Lemma \ref{exotic} and \ref{rho^i}, the maps $\rho$ and $\rho^1\times\cdots\times\rho^{p-1}$ are H-maps with respect to the exotic H-structure of $G$. Then for $\rho\vert_A\simeq(\rho^1\vert_{A_1})\vee\cdots\vee(\rho^{p-1}\vert_{A_{p-1}})\simeq(\rho^1\times\cdots\times\rho^{p-1})\vert_A$, the proof is done by Proposition \ref{universality}.
\end{proof}

\begin{corollary}
\label{decomp-hofib}
Suppose that $(G,p)$ is in Table \ref{(G,p)}. If $\rho\colon G\to\SU(\infty)$ is an H-map with respect to the standard H-structure, then
$$\hofib(\rho)\simeq\hofib(\rho^1)\times\cdots\times\hofib(\rho^{p-1}).$$
\end{corollary}


Put $d_k=k-1+r_k(p-1)$ for $k\in t_p(G)$. Similarly to $B_k$, we denote the factors of $\Omega^2\widetilde{\SU}(\infty)$ and $\rho$ corresponding to $k\in t_p(G)$ by $\Omega^2\widetilde{C}_k$ and $\rho^k$ respectively.

\begin{proposition}
\label{hofib-cohomology}
Suppose that $\rho^k\colon B_k\to\Omega^2\widetilde{C}_k$ is an isomorphism in cohomology of dimension $<2k-1+2r_k(p-1)$, then the cohomology of $\hofib(\rho^k)$ is given by
$$\widetilde{H}^*(\hofib(\rho^k))=\langle a_{2d_k},a_{2d_k+2(p-1)}\rangle\quad\text{for}\quad*<2d_k+4(p-1)$$
such that $a_{2n}$ transgresses to the suspension of $c_{n+1}$ modulo decomposables.
\end{proposition}

Define a map 
$$\Phi_k=a_{2d_k}\oplus a_{2d_k+2(p-1)}\colon[X,\hofib(\rho^k)]\to H^{2d_k}(X)\oplus H^{2d_k+2(p-1)}(X).$$

\begin{corollary}
\label{hofib-homotopy}
Assume the condition of Proposition \ref{hofib-cohomology}. If $\dim X<2d_k+4(p-1)$ and $[X,\hofib(\rho^k)]$ is a free $\Z_{(p)}$-module, then the map $\Phi_k$ is an injective homomorphism.
\end{corollary}

\begin{proof}
By Proposition \ref{hofib-cohomology}, the map $a_{2d_k}\times a_{2d_k+2(p-1)}\colon\hofib(\rho^k)\to K(\Z_{(p)},2d_k)\times K(\Z_{(p)},2d_k+2(p-1))$ is a rational $(2d_k+4(p-1))$-equivalence. Then by $\dim X<2d_k+4(p-1)$, $\Phi_k$ is an isomorphism after tensoring $\mathbb{Q}$, so since $[X,\hofib(\rho^k)]$ is a free $\Z_{(p)}$-module, the proof is completed.
\end{proof}

Suppose $G$ is a quasi-$p$-regular exceptional Lie group. Then by Table \ref{non-trivial1} and Proposition \ref{Samelson-3}, it remains to calculate the Samelson products $\langle\bar{\epsilon}_i,\bar{\epsilon}_j\rangle$ in $G$ for $(G,p,\{i,j\})$ in the following table.

\renewcommand{\arraystretch}{1.2}
\begin{table}[H]
\caption{}
\label{non-trivial2}
\centering
\begin{tabular}{lllll}
$\E_6$&$p=5$&$\{2,8\},\{5,5\}$&$p=7$&$\{2,6\},\{5,9\}$\\
$\E_7$&$p=11$&$\{2,8\},\{2,10\},\{8,10\},\{8,14\}$&$p=13$&$\{2,6\},\{2,12\},\{6,6\},\{6,8\},\{6,12\}$\\
$\E_8$&$p=11$&$\{2,20\},\{8,14\}$&$p=13$&$\{2,12\},\{8,18\}$\\
&$p=17$&$\{14,20\}$&$p=19$&$\{2,12\},\{2,18\},\{12,12\},\{14,24\}$
\end{tabular}
\end{table}

We denote the composite of the representation $\rho_\ell$ and the inclusion $\SU(N_\ell)\to\SU(\infty)$ by the same symbol $\rho_\ell$, where $N_\ell=27,56,248$ for $\ell=6,7,8$. For $(G,p,\{i,j\})=(\mathrm{E}_8,19,\{12,12\})$, it follows from Proposition \ref{Chern-1} that the condition of Proposition \ref{hofib-cohomology} dose not hold if $k=24$. However, there is $a_{2i-2}\in H^{2i-2}(\hofib(\rho_8^{24});\Z_{(p)})$ which transgresses to the suspension of $c_i$ for $i=42,60$. Then we define
$$\Phi_{24}=a_{82}\oplus a_{118}\colon[X,\hofib(\rho_8^{24})]\to H^{82}(X;\Z_{(p)})\oplus H^{118}(X;\Z_{(p)}).$$

\begin{proposition}
\label{[A,hofib]}
For $i,j\in t_p(G)$, let $X$ be the $(2d_{k(i,j)}+2(p-1))$-skeleton of $A_i\wedge A_j$, where $k(i,j)$ is the integer $k$ in Corollary \ref{pi(A)}. If $(G,p,\{i,j\})$ is in Table \ref{non-trivial2}, then the map $\Phi_{k(i,j)}$ is an injective homomorphism.
\end{proposition}

\begin{proof}
We first consider the case $(G,p,\{i,j\})\ne(\mathrm{E}_8,19,\{12,12\})$.
By Proposition \ref{pi(S)} and \ref{hofib-cohomology}, we see that $\pi_i(\hofib(\rho^k_\ell))\cong\Z_{(p)}$ and $\pi_{i+1}(\hofib(\rho^k_\ell))=0$ for $i=2d_k,2d_k+2(p-1)$, where $k=k(i,j)$, since they are in the stable range. Then since $X$ consists of cells in dimension $2d_k\bmod 2(p-1)$ and $\dim X\leq 2d_{k}+2(p-1)$ by definition, we see that $[X,\hofib(\rho_\ell^k)]$ is a free $\Z_{(p)}$-module by skeletal induction. Thus the proof is done by Corollary \ref{hofib-homotopy}.

We next consider the case $(G,p,\{i,j\})=(\mathrm{E}_8,19,\{12,12\})$. 
Since $c_{24}(\rho_8)=19\lambda x_{48}+\cdots$ for $\lambda\in\Z_{(p)}^\times$, we have
$$H^*(\hofib(\rho_8^{24});\Z/p)=\Z/p[a_{10},\mathcal{P}^1a_{10},\mathcal{P}^2a_{10},\mathcal{P}^3a_{10}]\otimes\Lambda(\beta\mathcal{P}^1a_{10}),\quad|a_{10}|=10$$
for $*<154$. Let $F$ be the 10-connective cover of $\hofib(\rho_8^{24})$. Then by the Serre spectral sequence of the homotopy fibration $K(\Z_{(p)},9)\to F\to\hofib(\rho_8^{24})$ and the Adem relation $\mathcal{P}^1\beta\mathcal{P}^2=2\beta\mathcal{P}^3+\mathcal{P}^3\beta$, we get
$$H^*(F;\Z/p)=\langle b_{82},\mathcal{P}^1b_{82}\rangle$$
for $*<154$. Then quite similarly to the above, we see that $[X,F]$ is a free $\Z_{(p)}$-module. Thus as in the proof of Corollary \ref{hofib-homotopy}, we see that the map 
$$\Phi'=b_{82}'\oplus b'_{118}\colon[X,F]\to H^{82}(X;\Z_{(p)})\oplus H^{118}(X;\Z_{(p)})$$
is injective, where $b_{82}',b_{118}'$ are mapped to non-zero multiples of $b_{82},\mathcal{P}^1b_{82}$ respectively by the mod $p$ reduction. Now by \cite[Lemma 2.1]{To2}, we may choose $b_{82}',b_{118}'$ such that the diagram
$$\xymatrix{[X,F]\ar[r]^(.31){\Phi'}\ar[d]&H^{82}(X;\Z_{(p)})\oplus H^{118}(X;\Z_{(p)})\ar[d]^{\times 19}\\
[X,\hofib(\rho_8^{24})]\ar[r]^(.4){\Phi_{24}}&H^{82}(X;\Z_{(p)})\oplus H^{118}(X;\Z_{(p)})}$$
commutes. Obviously, the left vertical arrow is an isomorphism. Then we obtain that the map $\Phi_{24}$ is injective, completing the proof.
\end{proof}


\section{Representations and Samelson products}

Consider an H-map $\rho\colon G\to\SU(\infty)$. Then there is an exact sequence:
$$\widetilde{K}(X)\cong[X,\Omega\SU(\infty)]\xrightarrow{\delta}[X,\hofib(\rho)]\to[X,G]\xrightarrow{\rho_*}[X,\SU(\infty)]$$
Suppose that $X=A\wedge B$ and consider the Samelson product $\langle\alpha,\beta\rangle$ in $G$ of maps $\alpha\colon A\to G,\beta\colon B\to G$. Since $\SU(\infty)$ is homotopy commutative we have $\rho_*(\langle\alpha,\beta\rangle)=0$, so there is $\gamma\in[X,\hofib(\rho)]$ which maps to $\langle\alpha,\beta\rangle$. Then we get:

\begin{lemma}
\label{im-delta}
The Samelson product $\langle\alpha,\beta\rangle$ is trivial if and only if $\gamma\in\mathrm{Im}\,\delta$.
\end{lemma}

This simple criterion is considered by Hamanaka and Kono \cite{HK} in the case that $\rho$ is the inclusion $\SU(n)\to\SU(\infty)$ for which $\hofib(\rho)$ is explicitly given by $\Omega\SU(\infty)/\SU(n)$. We apply Lemma \ref{im-delta} to determine (non-)triviality of the remaining Samelson products. Assume that $(G,p,\{i,j\})$ is in Table \ref{non-trivial2}. Put $X$ to be the $(2d_k+2(p-1))$-skeleton of $A_i\wedge A_j$. Then by  Corollary \ref{pi(A)}, there is only one $k(i,j)$ such that
$$[X,\hofib(\rho_\ell)]\cong[X,\hofib(\rho^{k(i,j)}_\ell)]$$
and by Corollary \ref{hofib-homotopy} and Proposition \ref{[A,hofib]}, we have identified the homotopy set on the right hand side. Quite similarly to \cite[Proposition 3.1]{HK}, we can prove the following.

\begin{proposition}
\label{Phi-delta}
In the situation of Proposition \ref{[A,hofib]}, we have
$$\Phi_{k(i,j)}\circ\delta=d_{k(i,j)}!\mathrm{ch}_{d_{k(i,j)}}\oplus(d_{k(i,j)}+p-1)!\mathrm{ch}_{d_{k(i,j)}+p-1}.$$
\end{proposition}

Let us calculate the image of $\Phi_{k(i,j)}\circ\delta$ explicitly. To choose generators of $\widetilde{K}(\Sigma A_i)$, we calculate the Chern character of the restriction $\bar{\rho}_\ell\colon \Sigma A\to B\SU(\infty)$ of $B\rho_\ell\colon BG\to B\SU(\infty)$. Note that
$$\mathrm{ch}_n(\bar{\rho_\ell})=\frac{(-1)^{n-1}}{(n-1)!}\iota^*(c_n(\rho_\ell))$$
by \eqref{Girard} where $\iota\colon\Sigma A\to BG$ denotes the inclusion. Then by Proposition \ref{Chern-1}, we have:

\begin{align*}
\mathrm{ch}(\bar{\rho}_6)=&3\Sigma y_3+\frac{12}{4!}\Sigma y_9-\frac{1}{2\cdot 5!}\Sigma y_{11}+\frac{5}{8\cdot 7!}\Sigma y_{15}+\frac{14}{8!}\Sigma y_{17}-\frac{5}{32\cdot 11!}\Sigma y_{23}\\
\mathrm{ch}(\bar{\rho}_7)=&6\Sigma y_3-\frac{1}{5!}\Sigma y_{11}+\frac{5}{4\cdot 7!}\Sigma y_{15}+\frac{126}{9!}\Sigma y_{19}-\frac{9}{2\cdot 11!}\Sigma y_{23}+\frac{319}{40\cdot 13!}\Sigma y_{27}-\frac{1229}{60\cdot 17!}\Sigma y_{35}\\
\mathrm{ch}(\bar{\rho}_8)=&30\Sigma y_3+\frac{15}{7!}\Sigma y_{15}+\frac{126}{11!}\Sigma y_{23}+\frac{165}{13!}\Sigma y_{27}+\frac{1820}{17!}\Sigma y_{35}+\frac{23001}{5\cdot 19!}\Sigma y_{39}-\frac{1746822}{5\cdot 23!}\Sigma y_{47}\\
&+\frac{15265250}{29!}\Sigma y_{59}
\end{align*}

\begin{remark}
Our expression of the Chern character of $\bar{\rho}_6$ differs from that of \cite{W2} since our choice of generators of $H^*(B\E_6)$ differs from that of \cite{W2}.
\end{remark}

We now choose generators of $\widetilde{K}(\Sigma A_i)$. If $r_i=1$, then $\Sigma A_i=S^{2i}$, implying $\widetilde{K}(\Sigma A_i)$ is a free $\Z_{(p)}$-module generated by a single generator $\eta_i$ such that 
$$\mathrm{ch}(\eta_i)=u_{2i}$$
where $u_m$ is a generator of $H^m(S^m)\cong\Z_{(p)}$. If $r_i=2$, then $\Sigma A_i=S^{2i}\cup e^{2i+2p-2}$, so there is a short exact sequence
\begin{equation}
\label{exact-K}
0\to\widetilde{K}(S^{2i+2p-2})\to\widetilde{K}(\Sigma A_i)\to\widetilde{K}(S^{2i})\to 0
\end{equation}
where $\widetilde{K}(S^{2n})\cong\Z_{(p)}$. If we put $\xi_i=\bar{\rho}_\ell\vert_{\Sigma A_i}$ for $i$ in Table \ref{non-trivial2}, 
then it is easily checked that
$$\mathrm{ch}(\xi_i)=au_{2i}+\cdots\quad(a\in\Z_{(p)}^{\times}).$$ 
So it follows from \eqref{exact-K} that $\widetilde{K}(\Sigma A_i)$ is a free $\Z_{(p)}$-module generated by $\xi_i$ and $\eta_i$ such that 
$$\mathrm{ch}(\eta_i)=u_{2i+2(p-1)}$$
where $\eta_i$ is explicitly given by the composite of the pinch map to the top cell $\Sigma A_i\to S^{2i+2p-2}$ and a generator of $\pi_{2i+2p-2}(B\SU(\infty))\cong\Z_{(p)}$. 

Since $\widetilde{K}(A_i)$ is torsion free, we have
$$\widetilde{K}(A_i\wedge A_j)\cong\Sigma^{-2}\widetilde{K}(\Sigma A_i)\otimes\widetilde{K}(\Sigma A_j).$$
Thus we obtain the following by Proposition \ref{Phi-delta}.

\begin{lemma}
\label{Phi-delta-K}
\begin{enumerate}
\item $\Phi_{k(i,j)}\circ\delta$ is surjective for $(G,p,\{i,j\})=(\E_7,11,\{2,10\})$, $(\E_7,13,\{2,12\})$, $(\E_7,13,\{6,8\})$, $(\E_8,19,\{2,18\})$.
\item For $(G,p,(i,j))=(\E_6,7,\{2,6\}),(\E_7,11,\{2,8\}),(\E_7,11,\{8,10\}),(\E_8, 19, \{2,12\})$, we have 
$$\mathrm{Im}\,\Phi_{k(i,j)}\circ\delta\supset p\cdot(H^{2i+2j+2p-4}(A_i\wedge A_j;\Z_{(p)})\oplus H^{2i+2j+4p-6}(A_i\wedge A_j;\Z_{(p)})).$$
\item $\mathrm{Im}\,\Phi_{k(i,j)}\circ\delta\mod p$ is generated by:
\begin{table}[H]
\label{gamma}
\centering
\begin{tabular}{lllll}
$\E_6$\\
$p=5$&$(2,8)$&$(0,0)$&$(5,5)$&$(0,*)$\\
$p=7$&$(2,6)$&$(3y_3\otimes y_{23}-y_{15}\otimes y_{11},4y_{15}\otimes y_{23})$&$(5,9)$&$(0,0)$\\
$\E_7$\\
$p=11$&$(2,8)$&$(5y_3\otimes y_{35}-2y_{23}\otimes y_{15},8y_{23}\otimes y_{35})$&$(8,14)$&$(0,*)$\\
$p=13$&$(2,6)$&$(5y_{3}\otimes y_{35}-3y_{27}\otimes y_{11},*)$&$(6,6)$&$(4y_{11}\otimes y_{35}+4y_{35}\otimes y_{11},-3y_{35}\otimes y_{35})$\\
&$(6,12)$&$(0,0)$\\
$\E_8$\\
$p=11$&$(2,20)$&$(0,0)$&$(8,14)$&$(0,0)$\\
$p=13$&$(2,12)$&$(-y_{3}\otimes y_{47}-3y_{27}\otimes y_{23},*)$&$(8,18)$&$(0,0)$\\
$p=17$&$(14,20)$&$(0,*)$\\
$p=19$&$(2,12)$&$(13y_{3}\otimes y_{59}+9y_{39}\otimes y_{23},15y_{39}\otimes y_{59})$&$(14,24)$&$(0,0)$
\end{tabular}
\end{table}
\end{enumerate}
\end{lemma}

\begin{corollary}
\label{Samelson-4}
If $i+j=p+1$ and $r_i+r_j=3$ for $i,j\in t_p(G)$, the Samelson product $\langle\bar{\epsilon}_i,\bar{\epsilon}_j\rangle$ is trivial.
\end{corollary}

\begin{proof}
For $i,j\in t_p(G)$, $i+j=p+1,r_i+r_j=3$ if and only if $(G,p,\{i,j\})=(\E_7,11,\{2,10\}),$ $(\E_7,13,\{2,12\}),(\E_7,13,\{6,8\}),(\E_8,19,\{2,18\})$. In these cases, we see from Lemma \ref{Phi-delta-K} that $\Phi_{k(i,j)}(\gamma)$ must be in $\mathrm{Im}\,\Phi_{k(i,j)}\circ\delta$, where $\gamma$ is a lift of $\langle\bar{\epsilon}_i,\bar{\epsilon}_j\rangle$. Thus by Lemma \ref{im-delta}, we obtain that the Samelson product $\langle\bar{\epsilon}_i,\bar{\epsilon}_j\rangle$ is trivial.
\end{proof}

We choose a lift $\gamma$ explicitly and calculate $\Phi_{k(i,j)}(\gamma)$ by generalizing a calculation in \cite{HK}. We may set the map $B\rho_\ell\colon BG\to B\SU(\infty)$ to be a fibration such that its fiber is $B\hofib(\rho_\ell)$. Let $\hat{\gamma}$ be the composite
$$\Sigma A_i\times\Sigma A_j\xrightarrow{\bar{\epsilon}_i'\times\bar{\epsilon}_j'}BG\times BG\xrightarrow{B\rho_\ell\times B\rho_\ell}B\SU(\infty)\times B\SU(\infty)\to B\SU(\infty)$$
where the last arrow is the Whitney sum. Then we have 
\begin{equation}
\label{gamma(c)}
(\hat{\gamma}^*(c_n))=\sum_{s+t=n}(\bar{\epsilon}_i')^*\circ B\rho_\ell^*(c_s)\otimes(\bar{\epsilon}_j')^*\circ B\rho_\ell^*(c_t).
\end{equation}
Put $n=d_k+1,d_k+p$ where $k=k(i,j)$. 
By Proposition \ref{Chern-1}, there exists $d_0\geq 0$ such that $p^{d_0}x_i \in \mathrm{Im}\,B\rho_{\ell}^*$ for each generator $x_i$ of $H^*(BG)$. Since $B\rho_\ell^*(c_n)$ has no linear part for a degree reason, there are a quadratic polynomial $Q\in\widetilde{H}^*(B\SU(\infty))^2$ and $R\in\widetilde{H}^*(B\SU(\infty))^3$ such that $B\rho_\ell^*(p^dc_n-Q+R)=0$ for some $d$.
Note that
$$\hat{\gamma}^*(p^dc_n-Q+R)=\hat{\gamma}^*(p^dc_n-Q)\in\pi^*(H^*(\Sigma A_i\wedge\Sigma A_j))$$
where $\pi\colon\Sigma A_i\times\Sigma A_j\to\Sigma A_i\wedge\Sigma A_j$ is the pinch map which has the canonical right inverse in cohomology. 

By definition, there is a strictly commutative diagram
$$\xymatrix{\Sigma A_i\vee\Sigma A_j\ar[r]\ar[d]^{\bar{\epsilon}'_i\vee\bar{\epsilon}'_j}&\Sigma A_i\times\Sigma A_j\ar[d]^{\hat{\gamma}}\\
BG\ar[r]^{B\rho_\ell}&B\SU(\infty).}$$
Let $\omega\colon\Sigma A_i\wedge A_j\to\Sigma A_i\vee\Sigma A_j$ be the Whitehead product. Since $I_\omega\simeq\Sigma A_i\vee\Sigma A_j$ for the mapping cylinder $I_\omega$ of $\omega$, we can apply a homotopy lifting property of the fibration $B\rho_\ell$ to get a commutative diagram
$$\xymatrix{I_\omega\ar[r]\ar[d]&\Sigma A_i\times\Sigma A_j\ar[d]^{\hat{\gamma}}\\
BG\ar[r]^(.4){B\rho_\ell}&B\SU(\infty)}$$
where the left and the upper arrows are equivalent to those of the above diagram and the upper one factors the pinch map $I_\omega\to C_\omega$ to the mapping cone. Then since $B\hofib(\rho_\ell)$ is a fiber of $B\rho_\ell$, we get a strictly commutative diagram
$$\xymatrix{\Sigma A_i\wedge A_j\ar[r]^(.65){\rm incl}\ar[d]&I_\omega\ar[r]\ar[d]&\Sigma A_i\times\Sigma A_j\ar[d]^{\hat{\gamma}}\\
B\hofib(\rho_\ell)\ar[r]&BG\ar[r]^(.4){B\rho_\ell}&B\SU(\infty).}$$
By adjointness of Whitehead products and Samelson products, we see that the adjoint of the left arrow is a lift of $\langle\bar{\epsilon}_i,\bar{\epsilon}_j\rangle$, so we fix $\gamma$ to be this map.

The last commutative diagram induces a commutative diagram
$$\xymatrix{H^{2n-1}(\Sigma A_i\wedge A_j)\ar[r]^\partial&H^{2n}(I_\omega,\Sigma A_i\wedge A_j)&H^{2n}(\Sigma A_i\times\Sigma A_j)\ar[l]_\cong\\
H^{2n-1}(B\hofib(\rho_\ell))\ar[r]^\partial\ar[u]_{\bar{\gamma}^*}&H^{2n}(BG,B\hofib(\rho_\ell))\ar[u]&H^{2n}(B\SU(\infty))\ar[u]_{\hat{\gamma}^*}\ar[l]_(.4){B\rho_\ell^*}}$$
where the upper $\partial$ is identified with the composite
$$H^{2n-1}(\Sigma A_i\wedge A_j)\xrightarrow{\cong}H^{2n}(\Sigma^2 A_i\wedge A_j)\xrightarrow{\pi^*}H^{2n}(\Sigma A_i\times\Sigma A_j)$$
for the projection $\pi\colon\Sigma A_i\times\Sigma A_j\to\Sigma^2A_i\wedge A_j$. Since there is $e_n\in H^{2n-1}(B\hofib(\rho_\ell))$ which transgresses to $p^dc_n-Q+R$, we have
$$\bar{\gamma}^*(e_n)=\Sigma^{-1}\circ(\pi^*)^{-1}(\hat{\gamma}^*(p^dc_n-Q+R))=\Sigma^{-1}\circ(\pi^*)^{-1}(\hat{\gamma}^*(p^dc_n-Q)).$$
Consider the Serre spectral sequence of the path-loop fibration of $B\SU(\infty)$. We have that the restriction of $e_n$ transgresses to $p^dc_n$ by naturality. 
Since the transgression induces an isomorphism between the modules of indecomposables of  $H^{2n-1}(\SU(\infty))$ and $H^{2n}(B\SU(\infty))$, the restriction of $e_n$ coincides with $p^d\sigma(c_n)$ where $\sigma(c_n)$ denotes the suspension of $c_n$. Thus we see that, in the Serre spectral sequence of the path-loop fibration of $B\hofib(\rho_\ell)$, the suspension of $e_n$ is equal to the pullback of $p^da_{2n-2}$ through the projection $\hofib(\rho_\ell)\to\hofib(\rho_\ell^k)$. Since $H^{2n-2}(A_i\wedge A_j)$ is a free $\Z_{(p)}$-module, we have the following:
$$\gamma^*(a_{2n-2})=\Sigma^{-2}\circ(\pi^*)^{-1}(\hat{\gamma}^*(p^dc_n-Q))/p^d.$$
Thus by combining this equation with \eqref{gamma(c)}, we obtain an explicit description of $\gamma^*(a_{2n-2})$.

We here give an example calculation of $\gamma^*(a_{2n-2})$ for $(G,p,\{i,j\})=(\E_6,5,\{5,5\})$ and $n=10$. By Proposition \ref{Chern-1}, we have $Q=\frac{1}{16}c_5^2,d=0$, and then
$$(\pi^*)^{-1}(\hat{\gamma}^*(c_{10}-Q))=(\pi^*)^{-1}((1- \tfrac{1}{8})(\epsilon_5)^*\circ\rho_6^*(c_5)\otimes(\epsilon_5')^*\circ\rho_6^*(c_5))=126\Sigma^2 y_9\otimes y_9,$$
so we get $\gamma^*(a_{18})=126y_9\otimes y_9$. Quite similarly to this calculation, we obtain:

\begin{lemma}
\label{Phi(gamma)-1}
$\Phi_{k(i,j)}(\gamma)\mod p$ for $(G,p,(i,j))$ is given by:
\renewcommand{\arraystretch}{1.2}
\begin{table}[H]
\label{gamma}
\centering
\begin{tabular}{lllll}
$\E_6$\\
$p=5$&$(2,8)$&$(3y_3\otimes y_{15},2y_3\otimes y_{23}+3y_{11}\otimes y_{15})$&$(5,5)$&$(y_9\otimes y_9,*)$\\
$p=7$&$(2,6)$&$(2 y_3\otimes y_{23}+4y_{15}\otimes y_{11},5y_{15}\otimes y_{23})$&$(5,9)$&$(2y_9\otimes y_{17},0)$\\
$\E_7$\\
$p=11$&$(2,8)$&$(7y_3\otimes y_{35}+6y_{23}\otimes y_{15},9y_{23}\otimes y_{35})$&$(8,10)$&$(0,0)$\\
&$(8,14)$&$(-3y_{15}\otimes y_{27},*)$\\
$p=13$&$(2,6)$&$(3y_{3}\otimes y_{35},*)$&$(6,6)$&$(2y_{11}\otimes y_{35}+2y_{35}\otimes y_{11},-2y_{35}\otimes y_{35})$\\
&$(6,12)$&$(4y_{35}\otimes y_{23},0)$\\
$\E_8$\\
$p=11$&$(2,20)$&$(5y_3\otimes y_{39},-y_3\otimes y_{59}+y_{23}\otimes y_{39})$&$(8,14)$&$(7y_{15}\otimes y_{27},y_{15}\otimes y_{47}+2y_{35}\otimes y_{27})$\\
$p=13$&$(2,12)$&$(-2y_{3}\otimes y_{47}-y_{27}\otimes y_{23},*)$&$(8,18)$&$(-6y_{15}\otimes y_{35},-9y_{15}\otimes y_{59}-4y_{39}\otimes y_{35})$\\
$p=17$&$(14,20)$&$(-9y_{27}\otimes y_{39},*)$\\
$p=19$&$(2,12)$&$(10y_{3}\otimes y_{59}+4y_{39}\otimes y_{23},13y_{39}\otimes y_{59})$&$(14,24)$&$(-5y_{27}\otimes y_{47},*)$\\
\end{tabular}
\end{table}
\end{lemma}

\begin{proposition}
\label{Samelson-5}
\begin{enumerate}
\item For $(G,p,\{i,j\})=(\E_6,7,\{2,6\})$, $(\E_7,11,\{2,8\})$, $(\E_7,11,\{8,10\})$, $(\E_8, 19, \{2,12\})$, the Samelson product $\langle\bar{\epsilon}_i,\bar{\epsilon}_j\rangle$ is trivial.
\item For the other $(G,p,\{i,j\})$ in the table of Lemma \ref{Phi(gamma)-1}, $\langle\bar{\epsilon}_i,\bar{\epsilon}_j\rangle$ is non-trivial.
\end{enumerate}
%
%
\end{proposition}

\begin{proof}
By Lemma \ref{Phi-delta-K} and \ref{Phi(gamma)-1}, we see that $\Phi_{k(i,j)}(\gamma)\in\mathrm{Im}\,\Phi_{k(i,j)}\circ\delta$ in the case of (1) and that $\Phi_{k(i,j)}(\gamma)\not\in\mathrm{Im}\,\Phi_{k(i,j)}\circ\delta$ in the case of (2). Then by Proposition \ref{[A,hofib]}, the proof is completed.
\end{proof}

\begin{proposition}
\label{Samelson-6}
The Samelson product $\langle\bar{\epsilon}_{12},\bar{\epsilon}_{12}\rangle$ in $\E_8$ at $p=19$ is trivial.
\end{proposition}

\begin{proof}
As above, we see that
\begin{align*}
\Phi_{24}\circ\delta(\xi_{12}\otimes\xi_{12})&\equiv 57(y_{23}\otimes y_{59}+y_{59}\otimes y_{23})+209y_{59}\otimes y_{59},\\
\Phi_{24}(\gamma)&\equiv 228(y_{23}\otimes y_{59}+y_{59}\otimes y_{23})+114y_{59}\otimes y_{59}
\end{align*}
modulo $19^2$. Similarly to Lemma \ref{Phi-delta-K}, we see that $\mathrm{Im}\,\Phi_{24}\circ\delta$ includes $19^2y_{23}\otimes y_{59},19^2y_{59}\otimes y_{23},19^2y_{59}\otimes y_{59}$, so $\Phi_{24}(\gamma)\in\mathrm{Im}\,\Phi_{24}\circ\delta$. Thus the proof is completed by Proposition \ref{[A,hofib]}.
\end{proof}

\begin{proof}
[Proof of Theorem \ref{main}]
Combine Corollary \ref{Samelson-1} and Proposition \ref{Samelson-2}, \ref{Samelson-3}, Corollary \ref{Samelson-4}, Proposition \ref{Samelson-5} and \ref{Samelson-6}.
\end{proof}

\end{document}